\documentclass[12pt]{amsart}

\usepackage{pst-all}

\usepackage{amsmath,amssymb,amsthm}    
\usepackage[colorlinks]{hyperref} 
\usepackage[abbrev,msc-links]{amsrefs}  
\usepackage{enumerate}
\usepackage{multirow}
\usepackage{enumitem}
\usepackage{float}
\usepackage{mathrsfs}
\usepackage{caption}
\usepackage{comment}

\usepackage[capitalise]{cleveref}

\newcommand{\E}{{\mathbb{E}}}

\newtheorem{theorem}{Theorem}

\newtheorem{lemma}{Lemma}

\theoremstyle{remark}

\renewcommand{\S}{{\mathbb{S}}}

\newcommand{\diam}{{\rm diam}\,}

\title{Convex bodies of constant width with exponential illumination number}
\author{A. Arman}
\address{Department of Mathematics, University of Manitoba, Winnipeg, MB, R3T 2N2, Canada}
\email{andrew0arman@gmail.com}
\thanks{The first author was supported by a postdoctoral fellowship of the Pacific Institute of Mathematical Sciences and the Department of Mathematics of the University of Manitoba}

\author{A.\ Bondarenko}
\address{Department of Mathematical Sciences, Norwegian University of Science and Technology, NO-7491 Trondheim, Norway}
\email{andriybond@gmail.com}
\thanks{The second author was supported in part by Grant 334466 of the Research Council of Norway.}

\author{A.\ Prymak}
\address{Department of Mathematics, University of Manitoba, Winnipeg, MB, R3T 2N2, Canada}
\email{prymak@gmail.com}
\thanks{The third author was supported by NSERC of Canada Discovery Grant RGPIN-2020-05357.}

\date{\today}



\begin{document}
	
	\begin{abstract}
		We show that there exist convex bodies of constant width in $\E^n$ with illumination number at least $(\cos(\pi/14)+o(1))^{-n}$, answering a question by G.~Kalai. Furthermore, we prove the existence of finite sets of diameter $1$ in $\E^n$ which cannot be covered by $(2/\sqrt{3}-o(1))^{n}$ balls of diameter $1$, improving a result of J.~Bourgain and J.~Lindenstrauss. 	
	\end{abstract}	
	
	\maketitle
	
	\section{Introduction}
	Let $\E^n$ denote the $n$-dimensional Euclidean space and $\S^{n-1}:=\{x:\|x\|=1\}$ be the unit sphere in $\E^n$. Consider a convex body $K$ in $\E^n$, i.e., a convex compact set with non-empty interior, and a point $x$ from the boundary $\partial K$ of $K$. We say that $x$ is illuminated by a direction $\xi\in\S^{n-1}$ if the half-line $\{x+t\xi:t\ge 0\}$ contains an interior point of $K$. The illumination number $I(K)$ is the minimal number of directions sufficient to illuminate all points $x\in\partial K$. A convex body is said to be of constant width $d$ if the distance between any two distinct parallel supporting hyperplanes of $K$ is $d$.

	{Illumination number of an arbitrary convex body $K\subset\E^n$ of constant width was studied by O.~Schramm~\cite{Schramm}, in particular he proved} 
		$$
	I(K)\le (\sqrt{3/2}+o(1))^{n}.
	$$
	However, it was not known if {there exists a constant width body $K$ with $I(K)\ge (1+\epsilon)^n$} for some $\epsilon>0$, see G.~Kalai's survey~\cite{K}*{Problem~3.3}. Note that in~\cite{K} the question is stated in terms of covering by smaller homothetic copies, but this is equivalent to illumination, for example see~\cite{BK} or~\cite{B}. {We answer Kalai's question in affirmative.}
	\begin{theorem}\label{thm:illum}
		For every positive integer $n$ there exists a convex body $K\subset\E^n$ of constant width satisfying $$I(K)\ge(\cos(\pi/14)+o(1))^{-n}.$$
	\end{theorem}
	To explain the main idea of the proof we need a few definitions. 
	For non-zero $x,y\in\E^n$, we use the notation $\theta(x,y):=\arccos(\tfrac{x\cdot y}{\|x\|\|y\|})$ for the angle between the directions of $x$ and $y$. In particular, if $x,y\in\S^{n-1}$, then $\theta(x,y)$ is the geodesic distance between $x$ and $y$ {on $\mathbb{S}^{n-1}$}. For $x\in\S^{n-1}$ and $0<\alpha<\pi$, the spherical cap of $\S^{n-1}$ centered at $x$ of (angular) radius $\alpha$ is $$C(x,\alpha):=\{y\in\S^{n-1}:\theta(x,y)\le \alpha\}.$$ Also, for fixed $x\in\S^{n-1}$ and $0<\alpha\le\pi/6$  define 
	$$
	Q(x,\alpha):=\{x\}\cup \{y\in \S^{n-1}:\|x-y\|=2\cos\alpha\}.
	$$
	The convex hull of $Q(x,\alpha)$ is the right circular cone with angle $\alpha$, apex $x$, {axis containing the origin, and} which is inscribed into $\S^{n-1}$.	Clearly, $\diam{Q(x,\alpha)}=2\cos\alpha$.
	Our main geometric observation is 
	\begin{lemma}\label{lemma:illum}
		Suppose $0<\alpha\le\pi/6$, $K$ is a convex body in $\E^n$ such that $\diam{K}=2\cos\alpha$ and for some $x\in\S^{n-1}$ we have $Q(x,\alpha)\subset K$. Then $x\in\partial K$ and any direction  $\xi\in\S^{d-1}$ illuminating $x$ satisfies $\xi\in C(-x,\tfrac\pi2-\alpha)$.
	\end{lemma}
	In other words, under the hypotheses of the lemma, any direction illuminating $x$ must belong to a certain ``well-controlled'' spherical cap.

	Fixing $\alpha=\frac\pi{14}$, we will choose a finite but ``large'' set
	$X\subset \S^{n-1}$. The points from $X$ will have sufficiently ``separated'' directions which will guarantee that the set
	
\begin{equation*}\label{eqn:Wdef}
		\mathcal{W}=\mathcal{W}(X):=\bigcup_{x\in X} Q(x, \alpha)
\end{equation*}	        
	satisfies $\diam{\mathcal{W}}=2\cos\alpha$. Next we will take an arbitrary convex body $K$ of {constant} width $2\cos\alpha$ that contains $\mathcal{W}$ (such a body always exists, see, e.g., \cite{E}*{Thm.~54, p.~126}). We will show that for a suitable randomly constructed $X$ only $O(n\log n)$ points from $X$ in $K$ can be illuminated simultaneously by a direction $\xi\in\S^{n-1}$. This will immediately imply Theorem 1. 
	
	The choice of $X\subset \S^{n-1}$ 
	will be provided by the following probabilistic lemma. 
		
\begin{lemma}\label{lemma:random}
	Suppose $0<\psi<\phi<\frac{\pi}{2}$ are fixed. Then for every positive 
	integer $n$ there exists $X\subset \mathbb{S}^{n-1}$ with cardinality $|X|\ge\left(\frac{1+o(1)}{\sin \phi}\right)^n$ such that
	\begin{itemize}
		\item[(a)] $\psi\leq \theta(x,y) \leq \pi-\psi$ for all  distinct $x,y\in X$, and
		\item[(b)] every point of~$\mathbb{S}^{n-1}$ is contained in at most $O(n\log n)$ spherical caps $C(x,\phi)$, $x\in X$.
	\end{itemize}
\end{lemma}
 	
 	Note {that the} probabilistic measure of each cap $C(x,\phi)$ on $\S^{n-1}$ is $(\sin\phi+o(1))^n$ (see for example~\cite{BW}*{Cor.~3.2(iii)}). Hence the {lower bound on $|X|$} in the lemma is optimal up to the $o(1)$ terms.
 	
	\cref{lemma:random} has the same spirit as P.~Erd{\H o}s and C.A.~Rogers's theorem~\cite{ER} concerning the covering of the space (see also the work~\cite{BW} by K.~B\"{o}r\"{o}czky and G.~Wintsche for the spherical case) and happens to be quite useful for constructing ``spread'' sets of fixed diameter. The lemma also readily implies the following result.
	\begin{theorem}\label{thm:balls}
		For every positive integer $n$ there exists a finite set of diameter $1$ in $\E^n$ which cannot be covered by $$(2/\sqrt{3}-o(1))^{n}$$ balls of diameter $1$.
	\end{theorem}
	This improves (note that $2/\sqrt{3}\approx 1.1547$) the result of J.~Bourgain and J.~Lindenstrauss~\cite{BL} who established that one needs at least $1.0645^n$ balls for large $n$. An exponential lower bound of $1.003^n$ was originally obtained by L.~Danzer~\cite{D}.

	Both problems, of illumination of {constant width convex bodies} and of {covering finite sets by balls of the same diameter}, are related to Borsuk's conjecture on partitioning a set into {parts} of smaller diameter. Namely, an upper bound in any of the former two problems implies the same upper bound in the latter. We refer to~\cite{K} for more details. Even in the three-dimensional case, it is not known what is the largest illumination number of constant width bodies. The conjectured value is 4, while the established upper bound is $6$, see~\cite{La}. In other small dimensions this problem was studied in~\cite{Be-Ki} and~\cite{BPR}. For broader context, and for other related questions, an interested reader may consult~\cite{BMP}, in particular, Ch.~3, Sect.~5.9 and Sect.~11.4.
	
	Shortly after the first version of this work was posted on arXiv, A.~Glazyrin~\cite{Gl} made a nice observation that a slight generalization of our construction in which the bases of the cones are chosen from a different sphere than the apexes allows one to improve the base of the exponent in \cref{thm:illum} from $\approx 1.026$ to $\approx 1.047$.
	

\section{Proofs}

We begin with geometric arguments and then proceed to the probabilistic ones concluding with the proofs of the theorems. For geometric lemmas, we define for $x\in \S^{n-1}$ and $\alpha \in (0,\frac{\pi}{6}]$, $R(x,\alpha):=\{y\in \S^{n-1}:\|x-y\|=2\cos\alpha\}=Q(x,\alpha)\setminus\{x\}$.
	

	\begin{proof}[Proof of \cref{lemma:illum}.]
		Clearly $R(x,\alpha)\ne\emptyset$, hence $x\in\partial K$ as otherwise $\diam{K}>2\cos\alpha$.
		
		Assume by contradiction that for some $\xi\in \S^{n-1}\setminus C(-x,\tfrac\pi2-\alpha)$ the point $x$ is illuminated by $\xi$. Consider a two-dimensional plane $H$ containing $x$, $\xi$ and the origin. Note that $H\cap R(x,\alpha)$ consists of the two points $y_1$, $y_2$ on $\S^1$ such that $\theta(y_k-x,-x)=\alpha$, $k=1,2$. On the other hand, as $\xi\in \S^{n-1}\setminus C(-x,\tfrac\pi2-\alpha)$, we have $\theta(\xi,-x)\ge\tfrac\pi2-\alpha$. Therefore, $\theta(y_k-x,\xi)\ge\tfrac\pi2$ for some $k\in\{1,2\}$. 
		
		Suppose $k=1$. 
		The distance from any point of the open half-line $\ell:=\{x+t\xi:t>0\}$ to $y_1$ is larger than $\|y_1-x\|=2\cos\alpha$, see Fig.~1. So $\ell$ contains no points from $K$ which contradicts our assumption. The case $k=2$ 
		is completely similar.
	\end{proof}
\begin{center}
	\psset{xunit=0.7cm,yunit=0.7cm}
\begin{pspicture}[showgrid=false](-6,-3)(6, 6)
	\psset{linewidth=1pt}
	\psline(0,0)(0,3)
	\uput{5pt}[90](0,3){$x$}
	\uput{6pt}[-90](0,0){$0$}
	\qdisk(0,0){2pt}\qdisk(0,3){2pt}
	\psline(-2,-2)(0,3)(2,-2)
	\qdisk(-2,-2){2pt}\qdisk(2,-2){2pt}
	\uput{6pt}[-90](-2,-2){$y_{1}$}
	\uput{6pt}[-90](2,-2){$y_{2}$}
	\uput{11pt}[-90](-0.3,2){$\alpha$}
	\uput{11pt}[-90](0.3,2){$\alpha$}
	\psline[linestyle=dashed](-5,5)(5,1)
	\psline[linewidth=1.5pt,arrows=->](0,3)(2.5,2.5)
	\uput{5pt}[90](2.2,2.6){$\vec{\xi}$}
	\psline[linestyle=dotted](0,3)(5,2)
	\uput{5pt}[0](5,2){$\ell$}
	\psset{linewidth=0.5pt}
	\psarc(0,3){1}{-112}{-68}
	\psline(-0.1,2.75)(-0.35,2.85)(-0.25,3.1)
\end{pspicture}

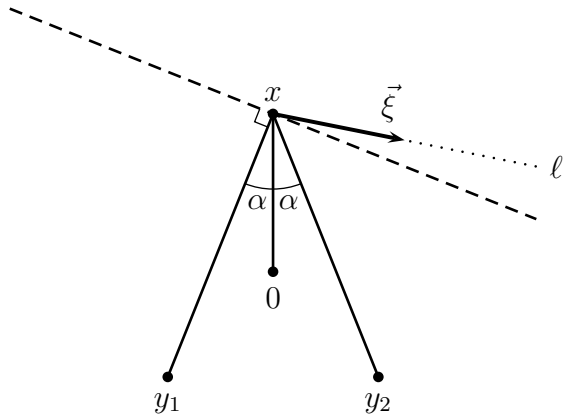
Figure 1. Section in the plane $H$
\end{center}

\vskip5mm

We next show that some necessary ``separation'' conditions on $X\subset \S^{n-1}$ guarantee that  $\diam{\mathcal{W}(X)}\le 2\cos\alpha$ (recall that $\mathcal{W}(X)=\bigcup_{x\in X} Q(x, \alpha)$).

\begin{lemma}\label{lemma:diam}
	Suppose $0<\alpha\le\pi/6$ and $X\subset \S^{n-1}$.
	\begin{itemize}
		\item[(i)] If $\theta(x,y)\le\pi-2\alpha$ for all $x,y\in X$, then  $\diam{X}\le 2\cos\alpha$.
		\item[(ii)] If $4\alpha\le \theta(x,y)\le\pi-6\alpha$ for all distinct $x,y\in X$, then $\diam{\mathcal{W}(X)}\le 2\cos\alpha$. 
	\end{itemize}
\end{lemma}
\begin{proof}
	We may switch to the geodesic distance by observing that $\theta(x,y)\le \pi-2\alpha$ for $x,y\in\S^{n-1}$ implies $\|x-y\|=2\sin\frac{\theta(x,y)}{2}\le 2\sin\frac{\pi-2\alpha}{2}=2\cos\alpha$. Thus, \textit{(i)} readily follows.
	
	For \textit{(ii)}, suppose $x,y\in X$, $x\ne y$. 
	Clearly, $\theta(x,y)\le \pi-6\alpha<\pi-2\alpha$. For any $u\in R(y,\alpha)$ we have $\theta(-y,u)=2\alpha$, so $\theta(u,x)\le \theta(u,-y)+\theta(-y,x)= 2\alpha+\pi-\theta(x,y)\le \pi-2\alpha$. Similarly, $\theta(v,y)\le\pi-2\alpha$ for any $v\in R(x,\alpha)$. Finally, suppose $u\in R(y,\alpha)$ and $v\in R(x,\alpha)$. Then $\theta(u,v)\le \theta(u,-y)+\theta(-y,-x)+\theta(v,-x)=4\alpha +\theta(x,y)\le \pi-2\alpha$ as required. 	
\end{proof}

%
%


\begin{proof}[Proof of~\cref{lemma:random}]
	It is enough to consider the case of sufficiently large $n$. 
	
	
	In order to satisfy (b) for a ``large'' set of cap centers, we start our proof by following the lines of the proof of~\cite{BW}*{Thm. 1.1}. 
	If $\Omega(\theta)$ denotes the probabilistic measure of a cap $C(x,\theta)$ on $\S^{n-1}$, let $N=\lceil\frac{8n \log n}{\Omega\left((1-\frac{1}{2n})\phi\right)}\rceil$
	and let the set $Y$ consist of $N$ points on $\mathbb{S}^{n-1}$, chosen independently and according to the uniform distribution on the sphere. Then, with probability $(1-o(1))$, $Y$ satisfies
	\[
	|\{y\in Y:x\in C(y,\phi)\}|\le 400 n\log n \quad\text{for every}\quad x\in\S^{n-1}.
	\] 
	We refer the reader to the proof of~\cite{BW}*{Case~1, p.~241} for full details and exact calculations.
	Note that by~\cite{BW}*{Cor.~3.2(iii)}, we get $\Omega\left((1-\frac{1}{2n})\phi\right)=\left((1+o(1))\sin\phi\right)^n$, and so $N=\left(\frac{1+o(1)}{\sin\phi}\right)^n$.  

	Next, using the alteration (deletion) method, we show that with positive probability there exists a large subset $X\subset Y$ satisfying property~(a). Indeed, consider all possible $N^2=\left(\frac{1+o(1)}{\sin \phi}\right)^{2n}$ pairs of points from $N$. Let 
	$$
	B:=B(Y):=\{(x,y):\theta(x,y)\not\in [\psi,\pi-\psi],\ x,y\in Y, x\ne y\}
	$$
	be the set of pairs of points from $Y$ not satisfying the property~(a). Note that a pair $(x, y)\in Y^2$, $x\ne y$, belongs to $B$ with probability  $p=2\Omega(\psi)=\left((1+o(1))\sin \psi\right)^n$, where we again use~\cite{BW}*{Cor.~3.2(iii)}. Thus the expected number of such ``bad'' pairs is $$\mathbb{E}(|B|)\leq p \cdot N^2=\left(\frac{(1+o(1))\sin \psi}{\sin \phi}\right)^n N.$$ 
	Since $\psi<\phi$ we can conclude that for large $n$, $\mathbb{E}(|B|)\leq \frac{N}{4}$. Now, using Markov's inequality, we deduce that with probability at least $\frac{1}{2}$ the number of pairs in $B$ does not exceed $\frac{N}{2}$. 
	
	To summarize, we showed that the random choice of $Y$ satisfies property~(b) with high probability, and $|B(Y)|<\frac{N}{2}$ with probability at least $1/2$. Selecting such a $Y$, we remove a point from each pair in $B(Y)$ to obtain $X\subset Y$ that has at least $\frac{N}{2}=\left(\frac{1+o(1)}{\sin \phi}\right)^n$ elements and satisfies both~(a) and~(b).
\end{proof}

\begin{proof}[Proof of \cref{thm:illum}.]
	Let $X$ be a collection of points whose existence is guaranteed by \cref{lemma:random} with parameters $\phi=\frac{6\pi}{14}+\epsilon$, $\psi=\frac{6\pi}{14}$, where $\epsilon>0$. By Lemma~\ref{lemma:diam}(ii) with $\alpha=\frac{\pi}{14}$, $\diam (\mathcal{W}(X))= 2\cos\frac{\pi}{14}$ (recall the definition of $\mathcal{W}(X)$). By~\cite{E}*{Thm.~54, p.~126}, there exists a convex body $K$ of constant width $2\cos\frac{\pi}{14}$ which contains $\mathcal{W}(X)$. Clearly, $\diam{K}=2\cos\frac{\pi}{14}$. 
	
	Now, if an $x\in X \subset \partial K$ is illuminated by a direction $\xi$, then $\xi \in C(-x, \phi)$ by Lemma~\ref{lemma:illum}. Note that if property~(b) of \cref{lemma:random} holds for $X$, then it also holds for the symmetric set $-X$. Therefore no direction $\xi$ can belong to more than $O(n\log n)$ caps $C(-x, \phi)$, $x\in X$. So every direction $\xi$ illuminates at most $O(n\log n)$ points from $X\subset \partial K$ and we need at least $\frac{|X|}{O(n\log n)}=\left(\frac{1+o(1)}{\sin(\frac{6\pi}{14}+\epsilon)}\right)^n=(\cos(\tfrac{\pi}{14}-\epsilon)+o(1))^{-n}$ directions to illuminate $\partial K$. Since $\epsilon>0$ can be arbitrarily small, this completes the proof.
\end{proof}

\begin{proof}[Proof of \cref{thm:balls}.]
	Let $X$ be a collection of points whose existence is guaranteed by \cref{lemma:random} with parameters $\phi=\frac{\pi}{3}+\epsilon$, $\psi=\frac{\pi}{3}$, where $\epsilon>0$. By \cref{lemma:diam}(i) with $\alpha=\frac{\pi}{6}$, $\diam{X}\leq 2\cos\frac{\pi}{6}=\sqrt{3}$.
	
	Now, if a ball with center $y$ of diameter $\sqrt{3}$ contains an $x\in X$, then $\frac{y}{\|y\|}\in C(x, \phi)$. Every such ball cannot contain more than $O(n\log n)$ points of $X$, by Lemma~\ref{lemma:random}(b). Therefore we need at least $\frac{|X|}{O(n\log n)}=\left(\frac{1+o(1)}{\sin(\frac{\pi}{3}+\epsilon)}\right)^n$
	balls of diameter $\sqrt{3}$ to cover the set $X$ of diameter at most $\sqrt{3}$. We conclude the proof by taking an appropriate homothetic copy of $X$ and noting that $\epsilon>0$ can be selected arbitrarily small.
\end{proof}

\end{document}